\documentclass[11pt,reqno]{amsart}
\usepackage{amscd,amsmath,amsopn,amssymb,amsthm,multicol}
\usepackage{hyperref}
\usepackage[all]{xy}
\usepackage{amscd}
\usepackage{graphicx}

\DeclareMathOperator{\Ad}{Ad}
\DeclareMathOperator{\ad}{ad}
\DeclareMathOperator{\Aut}{Aut}

\DeclareMathOperator{\Sym}{S}

\DeclareMathOperator{\Ric}{Ric}

\newcommand{\fr}{\mathfrak}
\newcommand{\al}{\alpha}
\newcommand{\be}{\beta}

\newcommand{\La}{\Lambda}

\newtheorem{theorem}{Theorem}

\newtheorem{prop}{Proposition}

\newtheorem{example}{Example}
\begin{document}
\title
{Invariant Einstein metrics  on  generalized flag manifolds with two isotropy summands}

\author{Andreas Arvanitoyeorgos and  Ioannis Chrysikos}

\address{University of Patras, Department of Mathematics, GR-26500 Rion, Greece}
\email{arvanito@math.upatras.gr}
\address{University of Patras, Department of Mathematics, GR-26500 Rion, Greece}
\email{xrysikos@master.math.upatras.gr}

\begin{abstract}
Let $M=G/K$ be a generalized flag manifold, that is the adjoint orbit of a  compact semisimple Lie group $G$.  We use the variational approach to find invariant Einstein metrics for all  flag  manifolds with two isotropy summands.  We also determine the nature of these Einstein metrics as critical points of the scalar curvature functional under fixed volume.

\medskip
\noindent
\thanks{The authors  were partially supported   
  by the C.~Carath\'{e}odory grant \#C.161 2007-10, University of
Patras.}

\medskip
\noindent 2000 {\it Mathematics Subject Classification.} Primary 53C25; Secondary 53C30, 22E46

\medskip
\noindent {\it Keywords}:  Einstein manifold, homogeneous space,  generalized flag manifold, isotropy representation, highest weight, Weyl's formula, bordered Hessian.

\end{abstract}


\maketitle

\section*{Introduction}
\markboth{Andreas Arvanitoyeorgos and  Ioannis Chrysikos}{Invariant Einstein metrics  on  generalized flag manifolds with two isotropy summands}

 A Riemannian metric $g$ on a manifold $M$ is called Einstein if the Ricci curvature is a constant multiple of the metric, i.e. ${\rm  Ric}_{g}=c\cdot g $,  for some $c\in\mathbb{R}$ ({\it Einstein equation}).  Einstein metrics form a special class of metrics on a given manifold $M$ (cf. \cite{Be}), and the existence question  is a fundamental problem in    Riemannian geometry.  The   Einstein equation is a non-linear second order  system of  partial differential equations and   general existence results are difficult to obtain. However, if the Riemannian manifold $(M, g)$ is compact, then an old result of Hilbert states that $g$ is an Einstein metric if and only if $g$ is a critical point of the scalar curvature functional $T : \mathcal{M}_{1}\to\mathbb{R}$ given by $T(g)=\int_{M}{\rm S}(g)d {\rm vol}_{g}$, on the set $\mathcal{M}_{1}$ of Riemannian metrics of unit volume.  This suggests a variational approach to finding Einstein metrics which, in the homogeneous case,  has lead  to several important existence and non-existence results mainly from the works of M. Wang, W. Ziller and  C. B\"ohm (\cite{Wa2}, \cite{Bom}, \cite{Bo}).  

    A usual strategy for constructing examples  of Einstein metrics  is to employ symmetry in order to reduce the Einstein equation into a more manageable system  of equations.   An important case  is when $M$ is a  homogeneous space, i.e.  when a Lie group  $G$ acts transitively on   $M$.  Many of the known examples of compact simply connected Einstein manifolds are homogeneous. 
    
     If $M=G/K$ is a homogeneous space with $G, K$   compact Lie groups, then  we can use  the variational approach to find Einstein metrics.  In this case  the $G$-invariant Einstein metrics on $M$ are precisely the critical points of $T$ restricted to $\mathcal{M}_{1}^{G}$, the set  of $G$-invariant metrics of volume 1. An alternative method is the direct computation of the Ricci curvature.   For both   cases,  since we are searching for $G$-invariant metrics on $M$,  the Einstein equation reduces to a system of non-linear algebraic equations which, in some  cases, can be solved explicity.  However a general  classification of all homogeneous  spaces that admit an Einstein metric, as well as the complete description of all $G$-invariant Einstein metrics on a given homogeneous Riemannian space $(M=G/K, g)$ is difficult.   For a detailed exposition on homogeneous Einstein manifolds we refer to Besse's book  \cite{Be}, and for more recent results  to the surveys \cite{LW}  and \cite{NRS}.

 An important class of homogeneous manifolds consists of the  adjoint  orbits of  compact connected semisimple Lie groups, also known as  {\it generalized flag manifolds}.  Let $G$ be a compact, connected and semisimple Lie group and let $\Ad : G\to\Aut(\fr{g})$ be the adjoint representation of $G$, where $\fr{g}$ denotes its Lie algebra.  A generalized flag manifold is a homogeneous space $M=G/K$ such that the isotropy group $K$ is the centralizer $C(S)$ of a torus $S$ in $G$.  This condition can be reformulated as follows: $M$ is the orbit of an element $\gamma_{o}\in\fr{g}$ under the action of the adjoint representation of $G$, i.e.  $M=\{{\rm Ad}(g)\gamma_{o} : g\in G\}\subset\fr{g}.$
   In fact, it can be shown that the stabilizer of this action $K=\{g\in G : {\rm Ad}(g) \gamma_{o}= \gamma_{o}\}$ is  the   centralizer  of  the torus $S_{\gamma_{o}} =\overline{\{{\rm exp}t\gamma_{o} : t\in\mathbb{R}\}}\subset G$  generated by the one-parameter subgroup ${\rm exp}t\gamma_{o}$ of $G$.   In particular, $K$ is connected  and the element $\gamma_{o}$ belongs to the center of the Lie algebra $\fr{k}$ of $K$  (cf. \cite{Be}).   If $S_{\gamma_{o}}=T$ is a maximal torus in $G$, then $K=C(S_{\gamma_{o}})=T$ and $M=G/T$  is called a {\it full flag manifold}.    Generalized flag manifolds have been classified in \cite{Bor} using the notion of the {\it painted Dynkin diagrams}. There is an infinite family for each of the classical Lie groups, and a finite number for each of the exceptional Lie groups (see also \cite{AA}).
  
Generalized flag manifolds have a   rich complex geometry.  It is known that they admit a finite number of invariant complex structures,  and that there is a one-to-one correspondence between invariant complex structures (up to a sign) and invariant K\"ahler-Einstein metrics (up to a scale) (cf. \cite{AP}, \cite{Bor}). The problem of finding (non K\"ahler) Einstein metrics on generalized flag manifolds has been first studied by D. V. Alekseevsky in \cite{Ale1}. For some of these spaces the standard metric is Einstein since they  appear in the work of M. Wang and W. Ziller (\cite{Wa1}), where they classified all normal homogeneous Einstein manifolds. In \cite{Kim} M. Kimura using the variational method of \cite{Wa2}   found all $G$-invariant Einstein metrics for all flag manifolds  for which the isotropy representation  decomposes into three inequivalent irreducible summands. In \cite{Sak}  Y. Sakane gave an explicit expression for the Ricci tensor of full flag manifolds of classical Lie groups, and   by use of   Gr\"obner bases theory he proved the excistence of invariant non K\"ahler-Einstein metrics  on certain full flag manifolds.  Finally, in  \cite{Arv}  the first author  found new $G$-invariant Einstein metrics on certain generalized flag manifolds with four  inequivalent irreducible summands, by using a Lie theoretic description   of  the Ricci tensor.

 In the present article we study generalized flag manifolds for which the isotropy representation decomposes into two inequivalent irreducible submodules.  Any such space admits a unique $G$-invariant complex structure (\cite{N}) and thus a unique K\"ahler-Einstein metric.   The  authors  classified  these spaces in  a recent  paper \cite{Chr} and proved that any such flag manifold is localy isomorphic to one of the spaces  listed in   Table 1.
\medskip
 \begin{center}
{\sc{Table 1.}} The generalized flag manifolds with two isotropy summands.  
  \end{center}  
  \smallskip
\begin{center}
$
   \begin{tabular}{|l|}
   \hline
 $B(\ell, m)=SO(2\ell +1)/U(\ell-m)\times SO(2m+1)$ \ $(\ell>0, m\geq 0, \ell-m\neq 1)$ \\
  $C(\ell, m)=Sp(\ell)/U(\ell-m)\times Sp(m)$ \ $(\ell>0, m>0)$\\
 $D(\ell, m)=SO(2\ell)/U(\ell-m)\times SO(2m)$\ $(\ell>0, m>0, \ell-m\neq 1)$ \\
    $G_{2}/U(2)$ \   ($U(2)$ is represented by the short root of $G_{2}$) \\
  $ F_{4}/SO(7)\times U(1)$ \\
 $F_{4}/Sp(3)\times U(1)$ \\
  $E_{6}/SU(6)\times U(1)$\\
     $E_{6}/SU(2)\times SU(5)\times U(1)$\\
  $E_{7}/SU(7)\times U(1)$\\
   $E_{7}/ SU(2)\times SO(10)\times U(1)$\\
   $E_{7}/ SO(12)\times U(1)$\\
  $E_{8}/E_{7}\times U(1)$\\
  $E_{8}/SO(14)\times U(1)$\\
  \hline
   \end{tabular}
   $
   \end{center}
  \bigskip
  
    In a recent work \cite{Ker},  W. Dickinson and M. Kerr, classified all  simply connencted homogeneous spaces $M=G/H$, where $G$  is a simple Lie group, $H$ is a connected and closed subgroup,  and the isotropy representation  decomposes into   two irreducible summands.  By using Theorem (3.1) of \cite{Wa2}  they counted the number of Einstein metrics.  In the present paper  by use of the variational method we find explicity the Einstein metrics for the flag manifolds presented in Table 1.  Next, by using the {\it bordered Hessian}  we examine the nature of these critical points (minima  or maxima).   After the present work has been completed, the authors were informed by Y. Sakane that solutions of Einstein equation have also been obtained in an unpublished work of  I. Ohmura   \cite{Om},  using the method of Riemannian submersions (cf. \cite{Be}).  
  
The paper is organized as follows: In Section 1 we recall some  facts about compact homogeneous spaces.  In Section 2 we study the structure of a generalized flag manifold $M=G/K$ of a compact semisimple Lie group $G$.  In Section 3 we use representation theory to compute the dimensions of the irreducible submodules of the isotropy representation corresponding to the flag manifolds presented in Table 1.  In Section 4, we solve the Einstein equation  by using the variational approach of \cite{Wa2} and prove the following:

 \medskip 
 { \sc{Theorem A.}}
{\it Let $M=G/K$ be  a generalized  flag manifold with two isotropy summands. Then $M$ admits  precisely two  $G$-invariant Einstein metrics.  One is K\"ahler and the other one is non-Kahler.  These  Einstein metrics are given  explicity in Theorem 2.}   
 
  \medskip
From the above theorem we exclude the following Hermitian symmetric spaces for which the standard metric is the unique (up to a scalar) $G$-invariant Einstein metric.
\smallskip
  \begin{center}
{\sc{ Table 2.}} Exceptions of the classification.  
\end{center}
\smallskip
\begin{center}
$ \begin{tabular}{|l|l|l|}
    \hline 
     Space & Case & Hermitian Symmetric Space \\
     \hline  
       $ B(\ell, m)$   &  $m=\ell-1$ &  $SO(2\ell+1)/U(1)\times SO(2\ell-1)$ \\
   $C(\ell, m)$ & $m=0$ & $Sp(\ell)/U(\ell)$ \\
   $B(\ell, m)$ & $m=\ell-1$ & $SO(2\ell)/U(1)\times SO(2\ell-2)$\\
      & $m=0$ & $SO(2\ell)/U(\ell)$\\
    \hline
   \end{tabular}$
   \end{center}
\medskip

 In Section 5 we compute the bordered Hessian of the scalar curvature functional with the constraint condition of volume 1 and characterize the nature of the  solutions   obtained in Theorem A.  In particular we show the following:

 \medskip 
 {\sc{Theorem B.}}
{\it Let $M=G/K$ be  a generalized  flag manifold with two isotropy summands. Then  the two $G$-invariant Einstein metrics on $M$  given in Theorem A, are both local minima   of the scalar curvature functional on the space $\mathcal{M}^{G}_{1}$. }

\section*{Acknowlegments}  
The second author wishes to thank Professor Yusuke Sakane for several useful discussions during his visit at the University of Patras.

\markboth{Andreas Arvanitoyeorgos and Ioannis Chrysikos}{Invariant Einstein metrics  on  generalized flag manifolds with two isotropy summands}
\section{Preliminaries}
\markboth{Andreas Arvanitoyeorgos and Ioannis Chrysikos}{Invariant Einstein metrics  on  generalized flag manifolds with two isotropy summands}

A Riemannian manifold $(M, g)$ is $G$-homogeneous if there is a closed subgroup $G$ of ${\rm Isom}(M, g)$ such that for any $p, q\in M$, there exists $g\in G$ such that $gp=q$.  Let $K=\{g\in G : gp=p\}$ be the isotropy subgroup corresponding to $p$.  Note that $K$ is compact since $K\subset O(T_{p}M)$, where $T_{p}M$ is the tangent space of $M$ at $p$.  Via the map $g\mapsto gp$ we identify the manifolds $M\cong G/K$.

 Let $M=G/K$ be a homogeneous space, where $G$ is a compact, connected and semisimple Lie group and $K$ is a closed subgroup of $G$.  Let $o=eK$  be the identity coset of $G/K$.  Several geometrical questions about  $M$  can be reformulated in terms of the pair $(G, K)$ and then in terms of the corresponding Lie algebras $(\fr{g}, \fr{k})$. In fact, there exists a one-to-one correspondence between $G$-invariant tensor fields of type $(p, q)$ on $M$  and tensors of the same type on the tangent space $T_{o}M$   which are invariant under the isotropy representation $\chi : K\to \Aut(T_{o}M)$ of $K$ on $T_{o}M$ (cf. \cite{Kob}).  For instance, left-invariant metrics on a Lie group are determined by an inner product on its Lie algebra, and   $G$-invariant Riemannian metrics $g$ on   $M=G/K$ are determined by an inner product on $\fr{g}/\fr{k}\cong T_{o}(M)$, with the additional requirement that the inner product is $\Ad(K)$-invariant.   
 
 Since the Lie group $G$ is semisimple and compact, the $\Ad(K)$-invariant Killing form $B(X, Y)={\rm tr}(\ad(X)\circ\ad(Y))$ of $\fr{g}$  is non-degenerate and negative definite.  Let $\fr{g}=\fr{k}\oplus\fr{m}$ be the orthogonal decomposition of $\fr{g}$ with respect to $-B$.    This is a reductive decomposition of $\fr{g}$, that is $\Ad(K)\fr{m}\subset\fr{m}$, and    the tangent space $T_{o}M$ is identified with $\fr{m}$.  Then $\Ad^{G}\big|_{K}=\Ad^{K}\oplus\chi$, where  $\Ad^{G}$ and $\Ad^{K}$ are the adjoint representations of $G$ and $K$ respectively.  It follows that the isotropy representation $\chi$   is equivalent to the adjoint representation   of $K$ restricted on $\fr{m}$, i.e.  $\chi(K) = {\rm Ad}^{K}\big|_{\fr{m}}$.  Therefore, a $G$-invariant metric on $G/K$ is determined by an ${\rm Ad}(K)$-invariant inner product $\langle \cdot \ , \cdot \rangle$ on $\fr{m}$. 
   
 Let $Q( . \  , \ .)$ be an $\Ad(K)$-invariant inner product on $\fr{m}$.   Consider the following $Q$-orthogonal $\Ad(K)$-invariant decomposition of $\fr{m}$ into its $\Ad(K)$-irreducible submodules, that is 
 \begin{equation}
 \fr{m}=\fr{m}_{1}\oplus\cdots\oplus\fr{m}_{q}.
 \end{equation}
 By using (1) we can parametrize the space of $G$-invariant metrics on $M$ so that any $G$-invariant metric $g$ on $M=G/K$ is determined by an inner product on $\fr{m}$ of the form 
 \begin{equation}
 \left\langle   \ , \ \right\rangle =x_{1}Q|_{\fr{m}_{1}}+\cdots+x_{q}Q|_{\fr{m}_{q}}, 
\end{equation}
 where $x_{i}>0$ for all $i$.
  Such a metric is called {\it diagonal}  since $\left\langle \ , \ \right\rangle$ is diagonal  with respect to $Q$.  If $\fr{m}_{i}$ and $\fr{m}_{j}$ are pairwise inequivalent representations, then the decomposition (1) is unique up to order. But if the modules $\fr{m}_{i}, \fr{m}_{j}$  are equivalent for some $i$ and $j$, then $\left\langle \fr{m}_{i}, \fr{m}_{j} \right\rangle$ does not necessarily vanish. For the examples in the present work we always have $\fr{m}_{i}\ncong\fr{m}_{j}$  for $i\neq j$ (as $\Ad(K)$-submodules). Also,  the dimensions $d_{i}=\dim{ \fr{m}_{i}}$ are independent of the chosen decomposition.
  
  Since $G$ is compact and $K$ is a closed subgroup of $G$ the homogeneous Riemannian manifold $(M=G/K, g)$ is compact.  Let   $S(g)$ denotes the scalar curvature of the metric $g$.   By a theorem of Bochner \cite{Boc} $S(g)$ in non-negative, and is zero if and only if the metric is flat \cite{All}.  Thus we are interested only in homogeneous Einstein metrics with positive scalar curvature, which is equivalent to $c>0$, where $\Ric_{g}=c\cdot g$. The Einstein metrics are the critical points of the total scalar curvature functional
   \[
   T(g)=\int_{M}S(g)d{\rm vol}_{g}
  \]
 on the space $\mathcal{M}_{1}$   of Riemannian  metrics of volume one.  Recall that the space $\mathcal{M}_{1}$ has a natural Riemannian metric, the $L^{2}$ metric, which is given by $ \left\|h\right\|_{g}^{2}=\int_{M}g(h, h)d{\rm vol}_{g}$, where $h$ is a symmetric 2-tensor (considered as a tangent vector at the metric $g$),  and $d{\rm vol}_{g}$ is the volume element of $g$.  Let $\mathcal{M}_{1}^{G}\subset\mathcal{M}_{1}$ denote the set of all $G$-invariant metrics of volume one on $M$, equipped with the restriction of the $L^2$ metric of $\mathcal{M}_{1}$.  This is also a Riemannian manifold, and since the isotropy representation of $M=G/K$ consists of pairwise inequivalent irreducible representations, $(\mathcal{M}_{1}^{G}, L^2)$ is flat with dimension equal to the number of irreducible summands (cf. \cite{Bom}, p.693).  Notice that on $\mathcal{M}_{1}^{G}$ we have $T(g)=S(g)$.  The  critical points of the restriction $S\big|_{\mathcal{M}_{1}^{G}} : {\mathcal{M}_{1}^{G}}\to \mathbb{R}$ are precisely the $G$-invariant Einstein metrics of volume one (cf. \cite{Be}, p.121).  In Section 4 we will use this variational approach to find Einstein metrics.
    
  For a fixed $Q$-orthogonal $\Ad(K)$-invariant decomposition  (1), the scalar curvature of the metric $(2)$ has a particularly simple expression, as shown in \cite{Wa2}.   Let $\{X_{\al}\}$ be a $Q$-orthogonal basis adapted to the decomposition of $\fr{g}$, i.e.  $X_{\al}\in \fr{m}_{i}$ for some $i$, and $\al<\be$ if $i<j$ with $X_{\al}\in \fr{m}_{i}$ and $X_{\be}\in\fr{m}_{j}$.  Set $A_{\al\be}^{\gamma}=Q([X_{\al}, X_{\be}], X_{\gamma})$ so that $[X_{\al}, X_{\be}]=\sum_{\gamma}A_{\al\be}^{\gamma}X_{\gamma}$, and   $[ijk]=\sum(A_{\al\be}^{\gamma})^{2}$, where the sum is taken over all indices $\al, \be, \gamma$ with $X_{\al}\in \fr{m}_{i}, X_{\be}\in\fr{m}_{j}$ and $X_{\gamma}\in\fr{m}_{k}$.   Notice that $[ijk]$ is indepedent of the $Q$-orthogonal bases chosen for $\fr{m}_{i}, \fr{m}_{j}$ and $\fr{m}_{k}$,  but it depends on the choise of the decomposition of $\fr{m}$.  Also, $[ijk]$ is nonnegative and symmetric in all three entries.   
  The set $\{X_{\al}/\sqrt{x_{i}} : X_{\al}\in\fr{m}_{i}\}$ is a $\left\langle  \ , \ \right\rangle$-orthogonal basis of $\fr{m}$. Then the scalar curvature of $\left\langle  \ , \ \right\rangle$ is given by 
  \begin{equation}
  S=\frac{1}{2}\sum_{i=1}^{q}\frac{d_{i}b_{i}}{x_{i}}-\frac{1}{4}\sum_{i, j, k}[ijk]\frac{x_{k}}{x_{i}x_{j}},
  \end{equation}
  where $d_{i}=\dim{\fr{m}_{i}}$, and $b_{i}$ is defined by $-B\big|_{\fr{m}_i}=b_{i}Q\big|_{\fr{m}_{i}}$  for all $i=1, \ldots, q$.

 \markboth{Andreas Arvanitoyeorgos and Ioannis Chrysikos}{Invariant Einstein metrics  on  generalized flag manifolds with two isotropy summands}
\section{The structure of flag manifolds}
\markboth{Andreas Arvanitoyeorgos and Ioannis Chrysikos}{Invariant Einstein metrics  on  generalized flag manifolds with two isotropy summands}

Let $G$ be a compact connected semisimple Lie group. We denote by $\fr{g}$ the corresponding Lie algebra and by $\fr{g}^{\mathbb{C}}$ its complexification. We choose a maximal torus  $T$ in $G$, and let $\fr{h}$ be the Lie algebra of $T$.  The complexification   $\fr{h}^{\mathbb{C}}$    is a Cartan subalgebra of $\fr{g}^{\mathbb{C}}$.  We denote by $R\subset(\fr{h}^{\mathbb{C}})^*$ the root system of $\fr{g}^{\mathbb{C}}$ relative to $\fr{h}^{\mathbb{C}}$,  and we consider the root space decomposition
\[
\fr{g}^{\mathbb{C}}=\fr{h}^{\mathbb{C}}\oplus\sum_{\al\in R}\fr{g}_{\al}^{\mathbb{C}},
\]
where by $\fr{g}_{\al}^{\mathbb{C}}=\mathbb{C}E_{\al}$ we denote  the 1-dimensional root spaces.

     Let $\Pi=\{\al_{1}, \ldots, \al_{\ell}\}$ \ $(\dim\fr{h}^{\mathbb{C}}=\ell)$ be a fundamental system of $R$.  We fix a lexicographic ordering on $(\fr{h}^{\mathbb{C}})^*$  and we denote by $R^{+}$ the set of positive roots.  It  is well known that for any $\al\in R$ we can choose root vectors $E_{\al}\in\fr{g}_{\al}^{\mathbb{C}}$ such that $B(E_{\al}, E_{-\al})=-1$ and $[E_{\al}, E_{-\al}]=-H_{\al}$, where    $H_{\al}\in\fr{h}^{\mathbb{C}}$ is determined by the equation $B(H, H_{\al})=\al(H)$, for all $H\in\fr{h}^{\mathbb{C}}$.  By using the last equation we  obtain a  natural isomorphism  between $\fr{h}^{\mathbb{C}}$ and the dual space $(\fr{h}^{\mathbb{C}})^{*}$.  The normalized root vectors $E_{\al}$   satisfy the relation
\[
 [E_{\al}, E_{\be}]=
 \left\{
\begin{array}{ll}
  N_{\al, \be}E_{\al+\be},  & \mbox{if} \ \ \al, \be,\al+\be\in R \\
  0, & \mbox{if} \ \ \al, \be\in R,  \al+\be\notin R
\end{array} \right.
\]
where $ N_{\al, \be}=N_{-\al, -\be}\in \mathbb{R}$ $(\al, \be\in R)$.
Then we obtain that (cf. \cite{Hel})
\[
\fr{g}=\fr{h}\oplus \sum_{\al\in R^{+}}(\mathbb{R}A_{\al}+\mathbb{R}B_{\al}),
\]
where $A_{\al}=E_{\al}+E_{-\al}, B_{\al}=\sqrt{-1}(E_{\al}-E_{-\al}), \al\in R^{+}$.  The complex conjugation $\tau$ on $\fr{g}^{\mathbb{C}}$ with respect to the compact real form $\fr{g}$ satisfies the relations $\tau(E_{\al})=E_{-\al}$ and $\tau(E_{-\al})=E_{\al}$.

We now assume that $G$ is simple. Let $\Pi_{K}$ be a subset of $\Pi$ and set
\[
\Pi_{M}=\Pi\backslash \Pi_{K}=\{\al_{i_{1}}, \ldots, \al_{i_{m}}\}, \qquad (1\leq i_{1}\leq \cdots\leq i_{m}\leq \ell).
\]
Let
\begin{equation}
R_{K}=R\cap\left\langle\Pi_{K} \right\rangle, \quad R_{K}^{+}=R^{+}\cap\left\langle\Pi_{K}\right\rangle, \quad R_{M}^{+}=R^{+}\backslash R_{K}^{+},
\end{equation}
where   $\left\langle\Pi_{K}\right\rangle$   denotes the set of roots generated by $\Pi_{K}$.   Then
\begin{equation}
\fr{p}=\fr{h}^{\mathbb{C}}\oplus\sum_{\al\in R_{K}}\fr{g}_{\al}^{\mathbb{C}}\oplus\sum_{\al\in R_{M}^{+}}\fr{g}^{\mathbb{C}}_{\al}
\end{equation}
is a parabolic subalgebra of $\fr{g}^{\mathbb{C}}$ (cf. \cite{Ale2}).

Let $G^{\mathbb{C}}$ be the simply connected complex simple Lie group whose Lie algebra is $\fr{g}^{\mathbb{C}}$ and $P$ the parabolic subgroup of $G^{\mathbb{C}}$ generated by $\fr{p}$. The   homogeneous space $G^{\mathbb{C}}/P$ is called {\it generalized flag manifold} (or {\it K\"ahler $C$-space}) and is a compact, simply connected complex manifold on which $G$ acts transitively.  Note that $K=G\cap P$ is a connected and closed subgroup of $G$.  The canonical embedding $G\to G^{\mathbb{C}}$ gives a diffeomorphism of a compact homogeneous space $M=G/K$ to a simply connected complex homogeneous space $G^{\mathbb{C}}/P$, i.e.  $G^{\mathbb{C}}/P\cong G/K$ and $M$ admits a $G$-invariant K\"ahler metric (cf. \cite{B}).  The intersection $\fr{k}=\fr{p}\cap\fr{g}\subset\fr{g}$    is the Lie subalgebra corresponding to $K$, given by $\fr{k}=\fr{h}\oplus\sum_{\al\in R_{K}^{+}}(\mathbb{R}A_{\al}+\mathbb{R}B_{\al})$.  By using (5)  we easily obtain the direct decomposition $\fr{p}=\fr{k}^{\mathbb{C}}\oplus\fr{n}$, where   $\fr{k}^{\mathbb{C}}$ is the complexification of $\fr{k}$ and $\fr{n}=\sum_{\al\in R_{M}^{+}}\fr{g}_{\al}^{\mathbb{C}}$ is the nilradical of $\fr{p}$.  

Let $\fr{m}$ be the linear subspace of $\fr{g}$ defined as:
\[
\fr{m}=\sum_{\al\in R_{M}^{+}}(\mathbb{R}A_{\al}+\mathbb{R}B_{\al}).
\]
Then  with respect to the Killing form $B$ we obtain the   reductive decomposition $\fr{g}=\fr{k}\oplus\fr{m}$ of $\fr{g}$ with $[\fr{k}, \fr{m}]\subset\fr{m}$.     
We define a complex structure $J$ on $\fr{m}\cong T_{o}M$  by
\[
JA_{\al}=B_{\al}, \quad JB_{\al}=-A_{\al} \qquad \ (\al\in R_{M}^{+}).
\]
This gives a $G$-invariant complex structure on $M=G/K$ and coincides with the canonical structure induced from the complex homogeneous space $G^{\mathbb{C}}/P$.

In the following we assume that  $\Pi_{K}=\Pi-\{\al_{i_{o}}\}$, that is $\Pi_{M}=\{\al_{i_{o}}\}$.  For a non-negative integer $n$, we set
\[
R^{+}(\al_{i}, n)=\big\{\al\in  R^{+} : \al=\sum_{i=1}^{\ell} m_{j}\al_{j}\in R^{+}, m_{i_{o}}=n\big\}, 
\]
and define  $\Ad(K)$-invariant subspaces $\fr{m}_{n}$ of $\fr{g}$ by $\fr{m}_{n}=\sum_{\al\in R^{+}(\al_{i}, n)}(\mathbb{R}A_{\al}+\mathbb{R}B_{\al})$.  Put $q={\rm max}\big\{m_{i_{o}} : \al=\sum_{j=1}^{\ell}m_{j}\al_{j}\in R^{+}\big\}.$  Then we obtain the decomposition  
\begin{equation}
\fr{m}=\sum_{n=1}^{q}\fr{m}_{n},  
\end{equation}
  and $R_{M}^{+}=\bigcup_{n=1}^{q}R^{+}(\al_{i}, n)$  (cf. \cite{It}).  We set $\fr{m}_{0}=\fr{k}$. Then for $n, m\in\{1, \ldots, q\}$ the following are true:
\begin{equation}
[\fr{k}, \fr{m}_{n}]\subset \fr{m}_{n}, \quad [\fr{m}_{n}, \fr{m}_{m}]\subset \fr{m}_{n+m}+\fr{m}_{|n-m|}, \quad  [\fr{m}_{n}, \fr{m}_{n}]\subset \fr{k}\oplus\fr{m}_{2n}.
\end{equation}
 Note that $\fr{m}_{n}$ are irreducible as $\Ad(K)$-modules and are inequivalent to each other.  Thus  (6) defines an irreducible decomposition of $\fr{m}$ and according to (2) the space of  $G$-invariant Riemannian  metrics on $M=G/K$ is given by
\[
 \Big\{x_{1}(-B)\big|_{\fr{m}_{1}}+\cdots+x_{q}(-B)\big|_{\fr{m}_{q}} : x_{1}>0, \ldots, x_{q}>0\Big\}.
 \]
    The following theorem describes the K\"ahler-Einstein metrics on the flag manifold $M=G/K$.  
 \begin{theorem}\cite{B}
 Let $M=G^{\mathbb{C}}/P=G/K$ be a generalized flag manifold and let  $B$ be the Killing form of $\fr{g}$. We assume that $\fr{g}$ admits a reductive decomposition given by (6).
  Then $M$ admits a $G$-invariant K\"ahler-Einstein metric defined  by
\[
g(\sum_{n=1}^{q}X_{n}, \sum_{n=1}^{q}Y_{n})= \sum_{n=1}^{q}n(-B(X_{n}, Y_{n})),   \qquad   (X_{n}, Y_{n}\in\fr{m}_{n}). 
\]  
  \end{theorem}

Let   $\fr{m}^{\mathbb{C}}$  and  $\fr{m}_{n}^{\mathbb{C}}$ be the complexifications of   $\fr{m}$ and $\fr{m}_{n}$ respectively.  These are complex linear subspaces of $\fr{g}^{\mathbb{C}}$ and we have that $\fr{m}^{\mathbb{C}}=\sum_{\al\in R_{M}^{+}}(\mathbb{C}E_{\al}+\mathbb{C}E_{-\al})$  and
$\fr{m}^{\mathbb{C}}_{n}=\fr{m}_{n}^{+}\oplus\fr{m}_{n}^{-}$, where $\fr{m}^{\pm}_{n}=\sum_{\al\in R^{+}(\al_{i}, n)}\mathbb{C}E_{\pm\al}$. Also  $\fr{k}^{\mathbb{C}}=\fr{h}^{\mathbb{C}}\oplus\sum_{\al\in R_{K}^{+}}(\mathbb{C}E_{\al}+\mathbb{C}E_{-\al}),$ and this is a complex reductive Lie algebra.  Thus we obtain the decomposition $\fr{k}^{\mathbb{C}}=
\fr{k}^{\mathbb{C}}_{s}\oplus \fr{z}(\fr{k}^{\mathbb{C}})$, where $\fr{k}^{\mathbb{C}}_{s}=[\fr{k}^{\mathbb{C}}, \fr{k}^{\mathbb{C}}]$ denotes the semisimple part of $\fr{k}^{\mathbb{C}}$, and $\fr{z}(\fr{k}^{\mathbb{C}})$ its center.  By \cite{Kim}, it is known that each 
  $\fr{m}^{\pm}_{n}$ is a complex irreducible $\ad_{\fr{g}^{\mathbb{C}}}(\fr{k}^{\mathbb{C}}_{s})$-submodule of $\fr{k}^{\mathbb{C}}_{s}$, where by $\ad_{\fr{g}^{\mathbb{C}}}$ we denote the adjoint representation of $\fr{g}^{\mathbb{C}}$.  This fact is equivalent with the irreducibility of the real $\ad_{\fr{g}}(\fr{k})$-submodules $\fr{m}_{n}$ of $\fr{k}$.
\medskip

\markboth{Andreas Arvanitoyeorgos and Ioannis Chrysikos}{Invariant Einstein metrics  on  generalized flag manifolds with two isotropy summands}
\section{Irreducible submodules}
\markboth{Andreas Arvanitoyeorgos and Ioannis Chrysikos}{Invariant Einstein metrics  on  generalized flag manifolds with two isotropy summands}

  \subsection{ Flag manifolds with two isotropy summands }
   
   In this paper we are interested in flag manifolds  $M=G/K$ for which the isotropy representation  $\chi : K\to\Aut(\fr{m})$ decomposes into exactly two real irreducible  inequivalent submodules, that is $\fr{m}=\fr{m}_{1}\oplus\fr{m}_{2}$.  In order to obtain  such flag manifolds  it is sufficient to set $R^{+}(\al_{i}, n)=0$ for $n\geq 3$.  
   
    Let $\mu=\sum_{i=1}^{\ell}m_{i}\al_{i}$ be the {\it highest root} of $R$, that is the unique  root such that for any other   root $\al=\sum_{i=1}^{\ell}c_{i}\al_{i}$  we have $c_{i}\leq m_{i}$  $(i=1, \ldots, \ell)$.  The positive coefficients   $m_{i}\in\mathbb{Z}$ are called {\it heights} of the simple roots $\al_{i}$.    In \cite{Chr} the authors   proved   that generalized flag manifolds with $\fr{m}=\fr{m}_{1}\oplus\fr{m}_{2}$ are in one-to-one correspondence with the sets $\Pi_{K}=\Pi-\{\al_{i_{o}}\}$ such that  the simple root $\al_{i_{o}}$ has height two, that is $m_{i_{o}}=2$.   We can present all these suitable pairs $(\Pi, \Pi_{K})$  graphically, by painting black the  simple root $\al_{i_{o}}$ in the Dynkin diagram of $G$.  The subdiagram of white roots determines the semisimple part of the Lie algebra of $K$.  These diagrams are given in Table 3.  
\newpage
  
 \begin{center}
{\sc{Table 3.}} Painted Dynkin diagrams of $M=G/K$ such that $\fr{m}=\fr{m}_1\oplus\fr{m}_2$   
\end{center}
 \smallskip
\begin{center}
\begin{tabular}{|c|c|c|c|}
\hline
\begin{picture}(20,30)(0,0)
\put(10, 12){\makebox(0,0){$G$}}\end{picture}
 &  \begin{picture}(30,30)(0,0)\put(10, 12){
\makebox(0,0){$( \Pi, \Pi^{}_K ) $}}\end{picture}
& \begin{picture}(30,30)(0,0)\put(15, 12){
\makebox(0,0){$K$}}\end{picture}
 &\begin{picture}(30,30)(0,0)\put(15, 12){
\makebox(0,0){\shortstack{$\dim {\frak m}^{}_1$ \\ \\
 $\dim {\frak m}^{}_2$}}}\end{picture}  
\\ 
\hline 

\begin{picture}(15,45)(0,0)
\put(10, 25){\makebox(0,0){$B^{}_\ell$}}\end{picture}
 &
\begin{picture}(160,45)(-15,-27)
\put(0, 0){\circle{4}}
\put(0,10){\makebox(0,0){1}}
\put(2, 0){\line(1,0){14}}
\put(18, 0){\circle{4}}
\put(20, 0){\line(1,0){10}}
\put(18,10){\makebox(0,0){2}}
\put(40, 0){\makebox(0,0){$\ldots$}}
\put(50, 0){\line(1,0){10}}
\put(60, -14){\makebox(0,0){$( 2 \leq p \leq \ell)$}}
\put(60, 0){\circle*{4.4}}
\put(60, 10){\makebox(0,0){$p$}}
\put(60, 0){\line(1,0){10}}
\put(80, 0){\makebox(0,0){$\ldots$}}
\put(90, 0){\line(1,0){10}}
\put(102, 0){\circle{4}}
\put(102, 10){\makebox(0,0){$\ell-1$}}
\put(103.5, 1.3){\line(1,0){15.5}}
\put(103.5, -1.3){\line(1,0){15.5}}
\put(115.5, -2.25){\scriptsize $>$}
\put(123.5, 0){\circle{4}}
\put(124, 10){\makebox(0,0){$\ell$}}
\end{picture}
 & 
\begin{picture}(110,45)(0,3)
\put(52, 15){\makebox(8,15){$U(p)\times SO(2(\ell-p)+1)$}}
 \end{picture}
 &
 \begin{picture}(85,45)(0,3)
\put(37, 17){
\makebox(5,15){\shortstack{$2p(2(\ell - p)+1)$ \\ 
\\ $p(p - 1)$}}}\end{picture}
\\  \hline

\begin{picture}(15,45)(0,0)
\put(10, 25){\makebox(0,0){$C^{}_\ell$}}\end{picture}
 &

\begin{picture}(160,45)(-15,-25)
\put(0, 0){\circle{4}}
\put(0,10){\makebox(0,0){1}}
\put(2, 0){\line(1,0){14}}
\put(18, 0){\circle{4}}
\put(20, 0){\line(1,0){10}}
\put(18,10){\makebox(0,0){2}}
\put(40, 0){\makebox(0,0){$\ldots$}}
\put(50, 0){\line(1,0){10}}
\put(60, -14){\makebox(0,0){$( 1 \leq p \leq \ell - 1 )$}}
\put(60, 0){\circle*{4.4}}
\put(60, 10){\makebox(0,0){$p$}}
\put(60, 0){\line(1,0){10}}
\put(80, 0){\makebox(0,0){$\ldots$}}
\put(90, 0){\line(1,0){10}}
\put(102, 0){\circle{4}}
\put(102, 10){\makebox(0,0){$\ell-1$}}
\put(107.2, 1.3){\line(1,0){14.6}}
\put(107.2, -1.3){\line(1,0){14.6}}
\put(103.46, -2.25){\scriptsize $<$}
\put(123.5, 0){\circle{4}}
\put(124, 10){\makebox(0,0){$\ell$}}
\end{picture}
 &

\begin{picture}(110,45)(5,15)\put(50, 20){
\makebox(8,35){$U(p)\times Sp(\ell-p)$}}\end{picture}
 &
 \begin{picture}(80,45)(5,15)\put(30, 20){
\makebox(15,35){\shortstack{$4 p (\ell - p)$ \\ \\ 
$p(p + 1)$}}}\end{picture}
\\  \hline

\begin{picture}(15,45)(0,0)
\put(10, 25){\makebox(0,0){$D^{}_\ell$}}\end{picture}
 &
\begin{picture}(160,40)(-15,-23)
\put(0, 0){\circle{4}}
\put(0,10){\makebox(0,0){1}}
\put(2, 0){\line(1,0){14}}
\put(18, 0){\circle{4}}
\put(20, 0){\line(1,0){10}}
\put(18,10){\makebox(0,0){2}}
\put(40, 0){\makebox(0,0){$\ldots$}}
\put(50, 0){\line(1,0){10}}
\put(60, -17){\makebox(0,0){$( 2 \leq p \leq \ell -2 )$}}
\put(60, 0){\circle*{4.4}}
\put(60, 10){\makebox(0,0){$p$}}
\put(60, 0){\line(1,0){10}}
\put(80, 0){\makebox(0,0){$\ldots$}}
\put(90, 0){\line(1,0){10}}
\put(102, 0){\circle{4}}
\put(103.7, 1){\line(2,1){10}}
\put(103.7, -1){\line(2,-1){10}}
\put(115.5, 6){\circle{4}}
\put(115.5, -6){\circle{4}}
\put(123.5, 14){\makebox(0,0){$\ell-1$}}
\put(120, -16){$\ell$}
\end{picture}
 &
\begin{picture}(110,45)(5,20)\put(50, 25){
\makebox(8,35){\shortstack{$U(p)\times SO(2(\ell-p))$}}}\end{picture}
 &
 \begin{picture}(80,45)(5,20)\put(30, 25){
\makebox(10,35){\shortstack{$4 p (\ell- p)$ \\ \\ 
$p(p - 1)$}}}\end{picture}
\\  \hline 
  
\hline \begin{picture}(15,50)(0,0)
\put(10, 20){\makebox(0,0){$E^{}_6$}}\end{picture}

&
\begin{picture}(160,35)(-25,3)

\put(15, 10){\circle{4}}
\put(17, 10){\line(1,0){16}}
\put(33, 10){\circle*{4.4}}
\put(33, 10){\line(1,0){16}}
\put(51,12){\line(0,1){14}}
\put(51, 10){\circle{4}}
\put(51, 28){\circle{4}}
\put(53,10){\line(1,0){14}}
\put(69,10){\circle{4}}
\put(71,10){\line(1,0){14}}
\put(87,10){\circle{4}}
\end{picture}
  & 
\begin{picture}(110,45)(0,0)\put(50, 20){
\makebox(10,10){$SU(5)\times SU(2)\times U(1)$}}\end{picture}
 &
 \begin{picture}(80,45)(0,0)\put(30, 20){
\makebox(10,10){\shortstack{$40$ \\ \\  \\ 
$10$}}}\end{picture}
\\  \hline
 
 \begin{picture}(15,50)(0,0)
\put(10, 20){\makebox(0,0){}}\end{picture}

&
\begin{picture}(160,35)(-25,5)
\put(15, 10){\circle{4}}
\put(17, 10){\line(1,0){14}}
\put(33, 10){\circle{4}}
\put(35, 10){\line(1,0){14}}
\put(51,12){\line(0,1){16}}
\put(51, 10){\circle{4}}
\put(51, 28){\circle*{4.4}}
 
\put(53,10){\line(1,0){14}}
\put(69,10){\circle{4}}
\put(71,10){\line(1,0){14}}
\put(87,10){\circle{4}}
\end{picture}
  & 
\begin{picture}(110,45)(0,0)\put(50, 17){
\makebox(10,10){$SU(6)\times U(1)$}}\end{picture}
 &
 \begin{picture}(80,45)(0,0)\put(30, 17){
\makebox(10,10){\shortstack{$40$ \\ \\  \\ 
$2$}}}\end{picture}
\\  \hline

\begin{picture}(15,35)(0,0)
\put(10, 13){\makebox(0,0){$E^{}_7$}}\end{picture}

&
\begin{picture}(160,35)(-25, 3)

\put(15, 10){\circle{4}}
\put(17, 10){\line(1,0){14}}
\put(33, 10){\circle{4}}
\put(35, 10){\line(1,0){14}}
\put(51,12){\line(0,1){14}}
\put(51, 10){\circle{4}}
\put(51, 28){\circle{4}}
\put(53,10){\line(1,0){14}}
\put(69,10){\circle{4}}
\put(71,10){\line(1,0){16}}
\put(87,9.5){\circle*{4.4}}
\put(87,10){\line(1,0){16}}
\put(105,10){\circle{4}}
\end{picture}
  & 
\begin{picture}(110,35)(0,0)\put(49, 10){
\makebox(10,10){$SO(10)\times SU(2)\times U(1)$}}\end{picture}
 &
 \begin{picture}(80,35)(0,0)\put(30, 10){
\makebox(10,10){\shortstack{$64$ \\ \\  \\ 
$20$}}}\end{picture}
\\  \hline
 
\begin{picture}(15,35)(0,0)
\put(10, 13){\makebox(0,0){}}\end{picture}

&
\begin{picture}(160,35)(-25, 3)

\put(15, 10){\circle*{4.4}}
\put(15, 10){\line(1,0){16}}
\put(33, 10){\circle{4}}
\put(35, 10){\line(1,0){14}}
\put(51,12){\line(0,1){14}}
\put(51, 10){\circle{4}}
\put(51, 28){\circle{4}}
\put(53,10){\line(1,0){14}}
\put(69,10){\circle{4}}
\put(71,10){\line(1,0){14}}
\put(87,10){\circle{4}}
\put(89,10){\line(1,0){14}}
\put(105,10){\circle{4}}
\end{picture}
  & 
\begin{picture}(110,35)(0,0)\put(50, 10){
\makebox(10,10){$SO(12)\times U(1)$}}\end{picture}
 &
 \begin{picture}(80,35)(0,0)\put(30, 10){
\makebox(10,10){\shortstack{$64$ \\ \\  \\ 
$2$}}}\end{picture}
\\  \hline

\begin{picture}(15,35)(0,0)
\put(10, 13){\makebox(0,0){}}\end{picture}

&
\begin{picture}(160,35)(-25, 3)

\put(15, 10){\circle{4}}
\put(17, 10){\line(1,0){14}}
\put(33, 10){\circle{4}}
\put(35, 10){\line(1,0){14}}
\put(51,12){\line(0,1){16}}
\put(51, 10){\circle{4}}
\put(51, 28){\circle*{4.4}}
\put(53,10){\line(1,0){14}}
\put(69,10){\circle{4}}
\put(71,10){\line(1,0){14}}
\put(87,10){\circle{4}}
\put(89,10){\line(1,0){14}}
\put(105,10){\circle{4}}
\end{picture}
  & 
\begin{picture}(110,35)(0,0)\put(50, 10){
\makebox(10,10){$SU(7)\times U(1)$}}\end{picture}
 &
 \begin{picture}(80,35)(0,0)\put(30, 10){
\makebox(10,10){\shortstack{$70$ \\ \\  \\ 
$14$}}}\end{picture}
\\  \hline

\begin{picture}(15,35)(0,0)
\put(10, 13){\makebox(0,0){$E^{}_8$}}\end{picture}

&
\begin{picture}(160,35)(-25, 3)

\put(15, 10){\circle{4}}
\put(17, 10){\line(1,0){14}}
\put(33, 10){\circle{4}}
\put(35, 10){\line(1,0){14}}
\put(51,12){\line(0,1){14}}
\put(51, 10){\circle{4}}
\put(51, 28){\circle{4}}
\put(53,10){\line(1,0){14}}
\put(69,10){\circle{4}}
\put(71,10){\line(1,0){14}}
\put(87,10){\circle{4}}
\put(89,10){\line(1,0){14}}
\put(105,10){\circle{4}}
\put(107,10){\line(1,0){16}}
\put(123,10){\circle*{4.4}}
 \end{picture}
  & 
\begin{picture}(110,35)(0,0)\put(50, 10){
\makebox(10,10){$E_{7}\times U(1)$}}\end{picture}
 &
 \begin{picture}(80,35)(0,0)\put(30, 10){
\makebox(10,10){\shortstack{$112$ \\ \\  \\ 
$2$}}}\end{picture}
\\  \hline
%
\begin{picture}(15,35)(0,0)
\put(10, 13){\makebox(0,0){}}\end{picture}
&
\begin{picture}(160,35)(-25, 3)
\put(15, 10){\circle*{4.4}}
\put(15, 10){\line(1,0){16}}
\put(33, 10){\circle{4}}
\put(35, 10){\line(1,0){14}}
\put(51,12){\line(0,1){14}}
\put(51, 10){\circle{4}}
\put(51, 28){\circle{4}}
\put(53,10){\line(1,0){14}}
\put(69,10){\circle{4}}
\put(71,10){\line(1,0){14}}
\put(87,10){\circle{4}}
\put(89,10){\line(1,0){14}}
\put(105,10){\circle{4}}
\put(107,10){\line(1,0){14}}
\put(123,10){\circle{4}}
 \end{picture}
  & 
\begin{picture}(110,35)(0,0)\put(50, 10){
\makebox(10,10){$SO(14)\times U(1)$}}\end{picture}
 &
 \begin{picture}(80,35)(0,0)\put(30, 10){
\makebox(10,10){\shortstack{$128$ \\ \\  \\ 
$28$}}}\end{picture}
\\  \hline
%
%
\begin{picture}(15,30)(0,0)
\put(10, 13){\makebox(0,0){$F^{}_4$}}\end{picture}
&
\begin{picture}(160,25)(30, -3)

\put(87,10){\circle{4}}
\put(89,10){\line(1,0){14}}
\put(105,10){\circle{4}}
\put(107, 11.3){\line(1,0){12.3}}
\put(107, 8.7){\line(1,0){12.3}}
\put(116, 7.75){\scriptsize $>$}
\put(124,10){\circle{4}}
\put(126,10){\line(1,0){16}}
\put(142,10){\circle*{4.4}}
\end{picture}
  & 
\begin{picture}(110,25)(0,2)\put(50, 10){
\makebox(10,10){$SO(7)\times U(1)$}}\end{picture}
 &
 \begin{picture}(80,25)(0,2)\put(30, 10){
\makebox(10,10){\shortstack{$16$ \\ \\  \\ 
$14$}}}\end{picture}
\\  \hline

\begin{picture}(15,30)(0,0)
\put(10, 13){\makebox(0,0){}}\end{picture}
& 

\begin{picture}(160,25)(30, -3)

 \put(87,10){\circle*{4.4}}
\put(88,10){\line(1,0){15}}
\put(105,10){\circle{4}}
\put(107, 11.3){\line(1,0){12.3}}
\put(107, 8.7){\line(1,0){12.3}}
\put(116, 7.75){\scriptsize $>$}
\put(124,10){\circle{4}}
\put(126,10){\line(1,0){14}}
\put(142,10){\circle{4}}

\end{picture}
  & 
\begin{picture}(110,25)(0,2)\put(50, 10){
\makebox(10,10){$Sp(3)\times U(1)$}}\end{picture}
 &
 \begin{picture}(80,25)(0,2)\put(30, 10){
\makebox(10,10){\shortstack{$28$ \\ \\  \\ 
$2$}}}\end{picture}
\\  \hline
\begin{picture}(15,30)(0,0)
\put(10, 13){\makebox(0,0){$G^{}_2$}}\end{picture}
&
\begin{picture}(160,25)(30, -3)

\put(105,10){\circle*{4.4}}
\put(106.8, 8.7){\line(1,0){12.9}}
\put(106.8,10){\line(1,0){16.1}}
\put(106.8,11.3){\line(1,0){12.9}}
\put(117, 7.75){\scriptsize $>$}
\put(125,10){\circle{4}}

\end{picture}
  & 
\begin{picture}(110,25)(0,2)\put(50, 10){
\makebox(10,10){$U(2)$ }}\end{picture}
 &
 \begin{picture}(80,25)(0,2)\put(30, 10){
\makebox(10,10){\shortstack{$8$ \\ \\  \\ 
$2$}}}\end{picture}
\\  \hline
\end{tabular} 
\end{center} 
\medskip
  
  \subsection{The classical flag manifolds}  
   
For the flag manifolds $B(\ell, m), C(\ell, m)$ and $D(\ell, m)$ the simplest method to compute the dimensions of the irreducible submodules is the straightforwad computation of the isotropy representation.  We will describe only the  case of the flag manifold $C(\ell, m)=Sp(\ell)\times U(\ell-m)\times Sp(m)$.  Results for the spaces $B(\ell, m)$ and $D(\ell, m)$  are obtained by a similar procedure.

 Set  $\ell-m=p$.  Let  $\mu_{p} : U(p)\to\Aut(\mathbb{C}^{p})$ and $\nu_{\ell} : Sp(\ell)\to \Aut(\mathbb{C}^{2\ell})$ be  the standard representations of  the Lie groups $U(p)$ and $Sp(\ell)$ respectively.  It is known (cf. \cite{Wa1}) that $\Ad^{U(n)}\otimes\mathbb{C}=\mu_{p}\otimes_{\mathbb{C}}\bar{\mu_{p}}$ and $\Ad^{Sp(\ell)}\otimes\mathbb{C}=\Sym^{2}\nu_{\ell}$, where $S^{2}$ is the second symmetric power of $\mathbb{C}^{2\ell}$.   Then
\begin{eqnarray*}
\Ad^{Sp(\ell)}\otimes \ \mathbb{C}\big|_{U(p)\times Sp(m)} &=& \Sym^{2}(\nu_{\ell}\big|_{U(p)\times Sp(m)}) =   \Sym^{2}(\mu_{p}\oplus\bar{\mu}_{p}\oplus\nu_{m}) \\
&=& \Sym^{2}\mu_{p}\oplus \Sym^{2}\bar{\mu}_{p}\oplus \Sym^{2}\nu_{m}\oplus(\mu_{p}\otimes\bar{\mu}_{p}) \\
&& \oplus (\mu_{p}\otimes \nu_{m})\oplus (\bar{\mu}_{p}\otimes \nu_{m}).
\end{eqnarray*}
 
The term $\Sym^{2}\nu_{m}$ corresponds to the complexified adjoint representation of $Sp(m)$ and  the term $\mu_{p}\otimes\bar{\mu}_{p}$ corresponds to the complexified adjoint representation of $U(p)$.  Therefore the complexified isotropy representation of $C(\ell, m)$ is given by  $(\mu_{p}\otimes \nu_{m})\oplus (\bar{\mu}_{p}\otimes \nu_{m})\oplus\Sym^{2}\mu_{p}\oplus \Sym^{2}\bar{\mu}_{p} .$  This is  the direct sum   of four  complex $\ad(\fr{k}^{\mathbb{C}})$-invariant inequivalent submodules of dimension $2pm$, $2pm$, $\binom{p+1}{2}$ and $\binom{p+1}{2}$, respectively.  The representations $\bar{\mu}_{p}\otimes \nu_{m}$ and $\mu_{p}\otimes \nu_{m}$ are conjugate to each other and the same holds for the pair $\Sym^{2}\bar{\mu}_{p}$ and  $\Sym^{2}\mu_{p}$.  Thus $\fr{m}$ decomposes into a direct sum of two real irreducible submodules $\fr{m}_{1}, \fr{m}_{2}$  of dimensions $4p(\ell-p)$ and $p(p+1)$ respectively.

\subsection{ Review of Representation Theory } 
In order to compute the dimensions of $\fr{m}_{1}, \fr{m}_{2}$ for the exceptional flag manifolds,  we need to review some facts from representation theory of complex semisimple Lie algebras and fix  notation.    We will first determine the  irreducible (finite-dimensional) representations of a complex semisimple Lie algebra, by use of the {\it highest weight}.

Let $\fr{g}$ be a complex semisimple Lie algebra of rank $\ell$, $\fr{h}$ a  Cartan subalgebra of $\fr{g}$, and $\fr{g}=\fr{h}\oplus\sum_{\al\in R}\fr{g}_{\al}$   the corresponding root space decomposition.  Let $\rho : \fr{g}\to{\rm End}(V)$  be a (finite-dimensional) representation     on the (complex) vector space $V$.    Then there is a  decomposition $V=\bigoplus_{\lambda }V_{\lambda }$, where $V_{\lambda }$ denotes the subspace of $V$ defined by 
\[
V_{\lambda }=\big\{ u\in V : \rho(H)u=\lambda (H)u \ \mbox{for all} \ H\in\fr{h}\big\}.
\]
 If $V_{\lambda }\neq 0$, then the linear form $\lambda \in\fr{h}^{*}$ is called a {\it weight} of $\rho$ and the eigenspace $V_{\lambda }$   is called the {\it weight space}.   The fact that $V$ is   finite dimensional implies that there is only a finite number of weights. It is well known that every weight is real-valued on the real form $\fr{h}_{\mathbb{R}}$ and is algebraically integral, that is $2\frac{(\lambda , \al)}{(\al, \al)}\in\mathbb{Z}$ for all $\al\in R$.  This property follows by restricting $\rho$ to copies of $\fr{sl}(2, \mathbb{C})$ lying in $\fr{g}$ and then using the representation theory of $\fr{sl}(2, \mathbb{C})$.  
 
 Let   $\Pi=\{\al_1, \ldots, \al_\ell\}$ be a system of simple roots for $R$. Then the elements $\{\Lambda _{1}, \ldots, \Lambda _{\ell}\}$,  $(\Lambda _{i}\in\fr{h}^{*})$ defined by $\frac{2(\Lambda _{i}, \al_{j})}{(\al_{j}, \al_{j})}=\delta_{ij}$,   $(i, j=1, \ldots, \ell)$ are called the {\it fundamental weights}. Under the identification of $\fr{h}$ and $\fr{h}^*$ via the Killing form, for any root $\al\in\fr{h}^{*}$ we consider the corresponding coroot $h_{\al}=\frac{2H_{\al}}{(H_{\al}, H_{\al})}=\frac{2\al}{(\al,\al)}\in\fr{h}$.  For the simple roots $\al_{i}$ we set $h_{i}=h_{\al_{i}}=\frac{2\al_{i}}{(\al_{i},\al_{i})}$. Then the fundamental weights $\Lambda _{i}$ satisfy $\Lambda _{i}(h_{j})=\delta_{ij}$, so they form a basis of $\fr{h}^*$ dual to the basis $\{h_{i}\}$  (with respect to the inner product    $( \ , \ )$   on $\fr{h}^{*}$).  The lattice of all integer combinations of the fundamental weights is called the  {\it weight lattice} $\Lambda $ of $\fr{g}$.    An important  weight is the sum of the fundamental weights $\delta=\sum_{i=1}^{\ell}\La_{i}$.  This  is also equal to half the sum of the positive roots of $R$.
 
  We shall  now give the relationship between the fundamental weights and the simple roots.   Since the fundamental weights form a basis of $\fr{h}^*$, there exist $c_{ij}\in\mathbb{C}$ such that $\al_{i}=\sum_{j=1}^{\ell}c_{ij}\Lambda _{j}$.  An easy calculation shows that $\al_{i}(h_{j})=c_{ij}$, therefore
 \[
 c_{ij}=\al_{i}(h_{j})=\al_{i}\Big(\frac{2\al_{j}}{(\al_{j}, \al_{j})}\Big)=\Big(\al_{i}, \frac{2\al_{j}}{(\al_{j}, \al_{j})}\Big)=\frac{2(\al_{i}, \al_{j})}{(\al_{j}, \al_{j})}=a_{ji},
 \]
where $a_{ji}$ is the transpose of the Cartan matrix $A=(a_{ij})=\big(\frac{2(\al_{i}, \al_{j})}{(\al_{i},\al_{i})}\big)$ of $\fr{g}$ (cf. \cite{Kna}).    Thus $\al_{i}=\sum_{j=1}^{\ell}a_{ji}\Lambda _{j}$, and the matrix expressing the simple roots as linear combinations of the fundamental weights is the transpose of the Cartan matrix. In particular, we note that all simple roots are integral combinations of fundamental weights and thus the lattice generated by the root system $R$ is contained in $\Lambda $.  

  Since $\fr{g}$ is a complex semisimple Lie algebra it is well known that every finite dimensional representation  of $\fr{g}$ is   a direct sum of irreducible (sub)representations.   Therefore, in order to study $\rho$ it is sufficient to study the irreducible representations of $\fr{g}$.  Fix a lexicographic ordering $R^{+}$ on $R$  (or equivalently,   a positive Weyl chamber of $\fr{g}$). This induces a partial ordering on all possible weights: $\lambda  >\mu$ if $\lambda -\mu$ is a sum of positive roots.  The maximal weight  $\lambda $ with respect to this ordering is called the {\it highest weight}.  Each (finite-dimensional) irreducible representation of $\fr{g}$ contains a highest weight $\lambda$ and will  denote this   representation by $\rho_{\lambda }$. The highest weight $\lambda$ characterizes completely  $\rho_{\lambda}$ since it  determines all of its properties, such as dimension.    Now, a weight $\lambda $ is called {\it dominant} if $(\lambda , \al)\geq 0$, for all simple roots $\al$, i.e.  if $\lambda $ lies in the closure of the Weyl chamber corresponding to $R^{+}$.  For example,  $\delta$  is a dominant weight.   We denote the collection of all dominants weights  by $\Lambda ^{+}$.  Equivalently, we can say that an element $\lambda \in\fr{h}^*$ lies in $\Lambda ^+$ if and only if $\lambda (h_{i})\in\mathbb{Z}$ and $\lambda (h_{i})\geq 0$ for all $i=1, \ldots, \ell$.   Dominant  weights are important since  any such weight arises as the highest weight of an irreducible representation of $\fr{g}$.  More specifically,  apart from equivalence,   irreducible finite-dimensional representations $\rho$ of a semisimple Lie algebra $\fr{g}$ are in one-to-one correspondence with the dominant weights $\lambda \in \Lambda ^+$.  The correspondence is that $\lambda $ is the highest weight of $\rho$, i.e.  $\rho\cong\rho_{\lambda }$.   The dimension of an irreducible representation of $\fr{g}$ is given by the following  formula due to H. Weyl.

\begin{prop}(\cite{Kna})
Let $\rho_{\lambda }$ be an irreducible representation of a complex semisimple Lie algebra $\fr{g}$ with highest weight $\lambda $.  Then  
\[
\dim_{\mathbb{C}}\rho_{\lambda } = \prod_{\al\in R^{+}}\Big(1+\frac{(\lambda, \al)}{(\delta, \al)}\Big).
\]
  \end{prop}

\subsection{ Exceptional flag manifolds }
 We will use  Proposition 1 to  compute the dimensions of the irreducible $\Ad(K)$-submodules $\fr{m}_{1}$ and $\fr{m}_{2}$  for the exceptional flag manifolds presented in Table 3.  We will only treat two cases, and the other are similar.  For the root systems of the exceptional Lie algebras we use the notation from \cite{AA}.
  
 \smallskip
{\bf Case of $\bold{G_2}$.}     We fix a system of simple roots to be  $\Pi=\{\al_{1}, \al_{2}\}=\{e_{2}-e_{3}, -e_2\},$ and let $R^{+}=\{\al_{1}, \al_{2}, \al_{1}+\al_{2}, \al_{1}+2\al_{2}, \al_{1}+3\al_{2}, 2\al_{1}+3\al_{2}\}$.  It is $(\al_{1}, \al_{1})=2$ and $(\al_{2}, \al_{2})=\frac{2}{3}$. The maximal root is  expressed in terms of simple roots as $\mu=2\al_1+3\al_2$.   The Cartan matrix $A=(a_{ij})$ of $G_2$  is given by (recall that  the Cartan matrix depends on the enumeration of $\Pi$)
\[
A =  \left( \begin{tabular}{cc}
 2  & -1 \\
-3  &  2 
\end {tabular}
\right)   
\]

\medskip
Consider  the painted Dynkin diagram
$
 \begin{picture}(29,0) (102, 2.9)
\put(105,5){\circle*{4.4}}
\put(105, 11.3){\makebox(0,0){$\al_1$}}
\put(106.8, 6.3){\line(1,0){13.2}}
\put(106.8, 5.3){\line(1,0){16.1}}
\put(106.8, 4.3){\line(1,0){12.9}}
\put(117, 3.2){\scriptsize $>$}
\put(125,5){\circle{4}}
\put(125.5, 11.3){\makebox(0,0){$\al_2$}}
\end{picture}
.$  \hskip 0.15cm It determines the generalized flag manifold $G_{2}/U(2)$, where $U(2)$ is represented by the short root $\al_2$. Thus  $R_{K}^{+}=\{\al_2\}$. The highest weights of the irreducible $\Ad(K)$-submodules $\fr{m}_{1}$ and $\fr{m}_{2}$ are given by $\lambda_{1}=\al_{1}+3\al_{2}$ and $\lambda_{2}=\mu$  respectively.  By using the transpose of the Cartan matrix, we obtain that
 \[ \al_1=2\La_1-3\La_2, \qquad \al_2=-\La_1+3\La_2, \]
where $\La_1, \La_2$  are the fundamental weights of $G_2$. Thus $\lambda_1=-\La_1+3\La_2$ and $\lambda_2=\La_1$.  Now we can use Weyl's formula, and obtain that $\dim_{\mathbb{C}}\fr{m}_{1}=\big(1+\frac{3}{1})=4$ and $\dim_{\mathbb{C}}\fr{m}_{2}=1$,  therefore $\dim_{\mathbb{R}}\fr{m}_{1}=8$ and $\dim_{\mathbb{R}}\fr{m}_{2}=2$.

\smallskip
{\bf Case of $\bold{F_4}$.}     Let  $\Pi=\{\al_{1}=e_{2}-e_{3}, \al_{2}=e_{3}-e_{4}, \al_{3}=e_{4}, \al_{4}=\frac{1}{2}(e_{1}-e_{2}-e_{3}-e_{4})\}$ be a system of simple roots.  Recall that $F_{4}$ contains long and short roots.  It is $(\al_{1}, \al_{1})=(\al_{2}, \al_{2})=2$ and  $(\al_{3}, \al_{3})=(\al_{4}, \al_{4})=1$.  The maximal root is expressed in terms of simple roots as $\mu=2\al_{1}+3\al_{2}+4\al_{3}+2\al_{4}$.   The Cartan matrix of $F_{4}$ is given by 
 \[
A =  \left( \begin{tabular}{cccc}
 2  & -1 &  0 &  0  \\
-1  &  2 & -1 &  0  \\
 0  & -2 &  2 & -1 \\
 0  &  0 & -1 &  2 
 \end {tabular}
\right)   
\]
Let $\{\Lambda_{1}, \Lambda_{2}, \Lambda_{3}, \Lambda_{4}\}$ be the fundamental weights of $F_{4}$.  Then by using the transpose of the above matrix we obtain that 
 \begin{eqnarray*}
\al_{1}&=& \ 2\Lambda _{1}-\Lambda _{2}, \\
\al_{2}&=&-\Lambda _{1}+2\Lambda _{2}-2\Lambda _{3}, \\
\al_{3}&=&-\La_{2}+2\La_{3}-\La_{4}, \\
\al_{4}&=&-\La_{3}+2\La_{4}. 
 \end{eqnarray*}

 Consider   the painted Dynkin diagram 
 \[
 \begin{picture}(150,20)(30, -3)

\put(87,5){\circle{4}}
\put(87, 11.5){\makebox(0,0){$\al_1$}}
\put(89,5){\line(1,0){14}}
\put(105,5){\circle{4}}
\put(105, 11.5){\makebox(0,0){$\al_2$}}
\put(107, 6.1){\line(1,0){12.3}}
\put(107, 4.5){\line(1,0){12.3}}
\put(116, 3.2){\scriptsize $>$}
\put(123.5,5){\circle{4}}
\put(123.5, 11.5){\makebox(0,0){$\al_3$}}
\put(126,5){\line(1,0){16}}
\put(142,5){\circle*{4.4}}
\put(142, 11.5){\makebox(0,0){$\al_4$}}
\end{picture}
\]
It detemines the generalized flag manifold   $M=G/K=F_{4}/SO(7)\times U(1)$.   The semisimple part of the isotropy subalgebra $\fr{k}^{\mathbb{C}}$ is the complex Lie algebra $\fr{so}(7, \mathbb{C})$ and its root system is generated by the set $\Pi_{K}=\{\al_{1}, \al_{2}, \al_{3}\}$.  In particular,  we obtain that
\[
R_{K}^{+}=\{\al_{1}, \al_{2}, \al_{3}, \al_{1}+\al_2, \al_{2}+\al_{3}, \al_{2}+2\al_{3}, \al_{1}+\al_{2}+\al_{3}, \al_{1}+\al_2+2\al_3, \al_1+2\al_2+2\al_3\}.
\]
The highest weight of the irreducible $\Ad(K)$-submodule  $\fr{m}_{1}$ is given by $\lambda_{1}=\al_{1}+2\al_{2}+3\al_3+\al_4$, and for $\fr{m}_{2}$ the corresponding highest weight is equal to the highest root, i.e.  $\lambda_{2}=\mu$.  By using the above expressions of simple roots in terms of the fundamental weights we  obtain that
$\lambda_{1}=\Lambda_3-\Lambda_4$ and $\lambda_2=\La_1$.  

  Now  we use  Weyl's formula.    Let  $\al=\sum_{i=1}^{3}c_{i}\al_{i}$ be a postive root of $R_{K}^{+}$. Since  $(\La_{i}, \al_j)\neq 0$ if and only if $i=j$, then for the  submodule $\fr{m}_{1}$ we have that
\begin{eqnarray*}
 (\lambda_{1}, \al) =(\Lambda_3-\Lambda_4, \al) &=& \sum_{i=1}^{3}c_{i}(\Lambda_3-\Lambda_4, \al_{i}) \\
 &=& c_{1}(\Lambda_3-\Lambda_4, \al_{1})+c_{2}(\Lambda_3-\Lambda_4, \al_{2})+c_{3}(\Lambda_3-\Lambda_4, \al_{3}) \\
 &=& c_{3}(\Lambda_3, \al_{3}) = c_{3}\frac{(\al_{3}, \al_{3})}{2}=\frac{c_{3}}{2},
 \end{eqnarray*}
 where $c_{3}\in\{0, 1, 2\}$.  For the weight $\delta_{K}=\La_1+\La_2+\La_3$ and for  $\al\in R_{K}^{+}$, we obtain that 
\begin{eqnarray*}
(\delta_{K}, \al) &=& (\La_1+\La_2+\La_3, \al) = \sum_{i=1}^{3}c_{i}(\La_1+\La_2+\La_3, \al_{i}) \\
&=& c_{1}(\La_1+\La_2+\La_3, \al_1)+c_{2}(\La_1+\La_2+\La_3, \al_2)+c_{3}(\La_1+\La_2+\La_3, \al_3) \\
&=& c_{1}(\La_1 , \al_{1})+ c_{2}(\La_2, \al_{2})+ c_{3}(\La_3, \al_{3}) = c_{1}\frac{(\al_{1}, \al_{1})}{2}+c_{2}\frac{(\al_{2}, \al_{2})}{2}+c_{3}\frac{(\al_{3}, \al_{3})}{2}\\
&=& c_{1}+c_{2}+\frac{c_{3}}{2}.
\end{eqnarray*}
From Weyl's formula it follows that
 \begin{eqnarray*}
 \dim_{\mathbb{C}}\fr{m}_{1}&=& \prod_{\al\in R_{K}^{+}}\Big(1+\frac{(\lambda_{1}, \al)}{(\delta_{K},\al)}\Big)\\
 &=&(1+\frac{1/2}{1/2})(1+\frac{1/2}{1+1/2})(1+\frac{1}{1+1})(1+\frac{1/2}{1+1+1/2})(1+\frac{1}{1+1+1})\\
 && (1+\frac{1}{1+2+1})=8,
 \end{eqnarray*}
so $\dim_{\mathbb{R}}\fr{m}_{1}=16$.
 
 For the  irreducible submodule $\fr{m}_{2}$ the calculations are simpler.  Since $\lambda_2=\Lambda_1$, for any positive root $\al=\sum_{i=1}^{3}c_{i}\al_{i}\in R_{K}^{+}$ it is
 \[
 (\lambda_{2}, \al)=(\La_1, \al) =\sum_{i=1}^{3}c_{i}(\La_1, \al_{i}) = c_{1}(\La_1, \al_{1})=c_{1}\frac{(\al_1, \al_{1})}{2}=c_1, \]
 where $c_{1}\in\{0, 1\}$.  
Thus, 
\begin{eqnarray*}
 \dim_{\mathbb{C}}\fr{m}_{2}&=& \prod_{\al\in R_{K}^{+}}\Big(1+\frac{(\lambda_{2}, \al)}{(\delta_{K},\al)}\Big)\\
 &=& (1+\frac{1}{1})(1+\frac{1}{1+1})(1+\frac{1}{1+1+1/2})(1+\frac{1}{1+1+1})(1+\frac{1}{1+2+1})=7, 
 \end{eqnarray*}
  so $\dim_{\mathbb{R}}\fr{m}_{2}=14$.

 We remark that it is also possible to use  Proposition 1 for the flag manifolds of a classical Lie group, but the computations are more complicated.

\markboth{Andreas Arvanitoyeorgos and Ioannis Chrysikos}{Invariant Einstein metrics  on  generalized flag manifolds with two isotropy summands}
\section{Invariant Einstein metrics on flag manifolds}
\markboth{Andreas Arvanitoyeorgos and Ioannis Chrysikos}{Invariant Einstein metrics  on  generalized flag manifolds with two isotropy summands}

Let $(M=G/K, g)$ be a Riemannian generalized flag manifold with $\fr{m}=\fr{m}_{1}\oplus\fr{m}_2$. In this  section we find the $G$-invariant Einstein metrics of $M$ by use of the variational method.    First we will  compute the scalar curvature of a $G$-invariant metric on $M$ by use of formula (3).  Since $G$ is simple we have  $b_{i}=1$. Thus, according to the relation (2), any $G$-invariant metric $g$ on $M$ is determined by two positive parameters $x_{1}, x_{2}$ and  has the form 
\begin{equation}
\left\langle \ , \ \right\rangle=x_{1}(-B)|_{\fr{m}_{1}}+x_{2}(-B)|_{\fr{m}_{2}}.  
 \end{equation}
 \begin{prop}
 Let $M=G/K$ be a generalized flag manifold with two isotropy summands and let $g$ be a  $G$-invariant Riemannian metric on $M$ given by (8). Then the scalar curvature   of  $g$  is given by:
 \[ 
S(g)=\frac{1}{2}\big(\frac{d_{1}}{x_{1}}+\frac{d_{2}}{x_{2}}\big)-\frac{1}{4}\big(t\frac{x_{2}}{x_{1}^{2}}+2t\frac{1}{x_{2}}\big),
\]
where $t=[112]\neq0$.

 \end{prop}
 \begin{proof}   We set $d_{1}=\dim\fr{m}_{1}$, and $d_{2}=\dim\fr{m}_{2}$.  In order to use (3) we need to find the triples $[ijk]$, where $i, j, k\in\{1, 2\}$.  Since $[ijk]$ is symmetric in all three entries it is $[111]=[222]=0$.  Relations (7) imply that
  \[
    [\fr{m}_{1}, \fr{m}_{2}]\subset\fr{m}_{1}, \quad [\fr{m}_{1}, \fr{m}_{1}]\subset\fr{m}_{2}\oplus\fr{k}, \quad[\fr{m}_{2}, \fr{m}_{2}]\subset\fr{k},
  \]
 so  a straightforward computation gives that
 \[
 [221]=[212]=[122]=0,
 \]
 thus the only non-zero triples are $[112]=[211]=[121]$. The result  now follows. 
 \end{proof}

Let $V(g)=x_{1}^{d_{1}}x_{2}^{d_2}$ be the volume  of a $G$-invariant metric $g$ on $M$  given by (8).  In order to determine the $G$-invariant Einstein metrics of  $M$  subject to the constrained condition   $V=1$, we need to study the critical points  of the restricted scalar curvature $S\big|_{\mathcal{M}_{1}^{G}}$.  According to the Lagrange multipliers method   a metric $g=(x_{1}, x_{2})\in \mathcal{M}_{1}^{G}$ is  a critical point  of  $S\big|_{\mathcal{M}_{1}^{G}}$ if and only if  it satisfies the   equation
\[
\nabla S(g)=c\nabla V(g),
\]
where $\nabla$ denotes the gradient field and $c$ is the Einstein constant.     Note  that since the irreducible submodules $\fr{m}_{1}$ and $\fr{m}_2$ are inequivalent, the space $\mathcal{M}_{1}^{G}$ is a   2-dimensional flat Riemannian manifold, i.e.  every point of  $\mathcal{M}_{1}^{G}$ has a neighborhood  locally isometric to an open set in $\mathbb{R}^{2}$.  

\begin{theorem}
Let $M=G/K$	 be a generalized flag manifold with two isotropy summands, i.e.  $\fr{m}=\fr{m}_{1}\oplus\fr{m}_{2}$ with $d_{i}= \dim\fr{m}_{i}$  $(i=1,2)$.  Then $M$ admits two $G$-invariant Einstein metrics. One is K\"ahler given by $x_{1}=1, x_{2}=2$, and the other is non K\"ahler   given by $x_{1}=1,  x_{2}=\displaystyle\frac{4d_{2}}{d_{1}+2d_{2}}$. 
\end{theorem}   
\begin{proof}
 Set $\tilde{S}=S-c(x_{1}^{d_{1}}x_{2}^{d_{2}}-1)$.  The volume condition is given by $\displaystyle\frac{\partial\tilde{S}}{\partial c}=0$. Thus a   $G$-invariant Einstein metric of volume one      is a solution  of the system
 \[
 \frac{\partial{\tilde{S}}}{\partial x_{1}} = 0 \ ,  \quad \frac{\partial{\tilde{S}}}{\partial x_{2}} = 0,
 \]
 which is equivalent to
\begin{equation}
 \begin{array}{r}
-\displaystyle\frac{d_{1}}{2x_{1}^{2}}+\displaystyle\frac{tx_{2}}{2x_{1}^{3}}-c d_{1}x_{1}^{d_{1}-1}x_{2}^{d_{2}} = 0  \\
\displaystyle\frac{t-d_{2}}{2x_{2}^{2}}-\displaystyle\frac{t}{4x_{1}^{2}}-c d_{2}x_{1}^{d_{1}}x_{2}^{d_{2}-1} = 0 
\end{array} \Bigg\} 
\end{equation} 
System (9) reduces  to the following polynomial equation
\begin{equation}
2td_{1}x_{1}^{2}-2d_{1}d_{2}x_{1}^{2}-td_{1}x_{2}^{2}+2d_{1}d_{2}x_{1}x_{2}-2td_{2}x_{2}^{2}=0.
\end{equation}

 Next, we need to find the number $t=[112]$. By Theorem 1, the space  $M=G/K$ admits a unique   K\"ahler-Einstein metric given by  $ x_{1}=1, x_{2}=2$, so substituting these values to (10) we obtain the equation
\[
2td_{1}-2d_{1}d_{2}-4td_{1}+4d_{1}d_{2}-8td_{2}=0,
\]
 from which $t=\displaystyle\frac{d_{1}d_{2}}{d_{1}+4d_{2}}$.
 We substitute this number to equation (10) and  normalize     $x_{1}=1$, to obtain the   equation
\[
 d_{1}d_{2}(x_{2}-2)\big(d_{1}x_{2}+2d_{2}(x_{2}-2)\big)=0,
\]
whose solutions   are $x_{2}=2$ and $x_{2}=\displaystyle\frac{4d_{2}}{d_{1}+2d_{2}}$.  The first solution determines the K\"ahler-Einstein metric, and the second solution determines the   non K\"ahler-Einstein metric on $M$.  
\end{proof}

\begin{example}
\textnormal{ Consider the family $C(\ell, m)$ and set $m=\ell-1$.  Then we obtain the generalized flag manifold $M=Sp(\ell)/U(1)\times Sp(\ell-1)$ which is the complex projective space $\mathbb{C}P^{2\ell-1}$. The painted Dynkin diagram  is given by
\[
\begin{picture}(100,15)(-15,-5)
\put(0, 0){\circle*{4}}
\put(0,10){\makebox(0,0){1}}
\put(2, 0){\line(1,0){14}}
\put(18, 0){\circle{4}}
\put(20, 0){\line(1,0){10}}
\put(18,10){\makebox(0,0){2}}
\put(40, 0){\makebox(0,0){$\ldots$}}
\put(50, 0){\line(1,0){10}}
\put(62.5, 0){\circle{4}}
\put(62.5, 10){\makebox(0,0){$\ell-1$}}
\put(66.5, 1){\line(1,0){14.6}}
\put(66.5, -1.1){\line(1,0){14.6}}
\put(64, -2.1){\scriptsize $<$}
\put(83, 0){\circle{4}}
\put(83.5, 10){\makebox(0,0){$\ell$}}
\end{picture}
\]
Form Table 3  we have that $d_{1}=\dim\fr{m}_1=4(\ell-1)$ and $d_2=\dim\fr{m}_2=2$.  Any $Sp(\ell)$-invariant metric $\left\langle   \ , \ \right\rangle$ on $\mathbb{C}P^{2\ell-1}$ is determined by two positive parameters $x_{1}, x_{2}$, so it is given by
\[
\left\langle   \ , \ \right\rangle=(-B)\big|_{\fr{m}_{1}}+\frac{x_2}{x_1}(-B)\big|_{\fr{m}_{2}},
\]
where $B$ is the Killing form of $Sp(\ell)$. From Theorem 2  we obtain that the (non K\"ahler) $Sp(\ell)$-invariant Einstein metric on $\mathbb{C}P^{2\ell-1}$ is given by $\left\langle \ , \ \right\rangle = 1 (-B)\big|_{\fr{m}_{1}}+\frac{2}{\ell}(-B)\big|_{\fr{m}_2}$.}
\end{example}

 Note  that the same result has also been   obtained  by W. Ziller  \cite{Zi}   by using the method of Riemannian submersions. He proved that the complex projective space $\mathbb{C}P^{2n+1}=Sp(n+1)/U(1)\times Sp(n)$ admits two $Sp(n+1)$-invariant Einstein metrics explicity given by  $\left\langle \ , \ \right\rangle=1(-B)\big|_{\fr{m}_{1}}+2p(-B)\big|_{\fr{m}_{2}}$, where $p=1$  or $p=\frac{1}{n+1}$. The   value $p=1$ gives the K\"ahler-Einstein metric, and the   value $p=\frac{1}{n+1}$  gives  the  non K\"ahler-Einstein metric.

\markboth{Andreas Arvanitoyeorgos and Ioannis Chrysikos}{Invariant Einstein metrics  on  generalized flag manifolds with two isotropy summands}
\section{Characterization of the constrained critical points of $S$}
\markboth{Andreas Arvanitoyeorgos and Ioannis Chrysikos}{Invariant Einstein metrics  on  generalized flag manifolds with two isotropy summands}
 
 We will use  a well known criterion (of  second order partial derivatives) for minima  and maxima  of smooth functions to show that both    Einstein metrics obtained  in Theorem 2 are local minima    of $S\big|_{\mathcal{M}_{1}^{G}}$. In particular,  we use the {\it bordered Hessian} $H$ of $S(g)$ restricted to the space $\mathcal{M}_{1}^{G}$ of $G$-invariant metrics with volume one.   This is the $3\times 3$ real symmetric matrix
\[
H= \left( \begin{tabular}{ccc}
$0$  & $-\displaystyle\frac{\partial V}{\partial x_1}$ & $-\displaystyle\frac{\partial V}{\partial x_2}$\\\\ 
$-\displaystyle\frac{\partial V}{\partial x_1}$ & $\displaystyle\frac{\partial^2 \tilde{S}}{\partial x_1^2}$ & $\displaystyle\frac{\partial^2\tilde{S}}{\partial x_1\partial x_2}$ \\\\ 
$-\displaystyle\frac{\partial V}{\partial x_2}$ & $\displaystyle\frac{\partial^2\tilde{S}}{\partial x_1\partial x_2}$ & $\displaystyle\frac{\partial^2\tilde{S}}{\partial x_2^2}$
\end{tabular}\right)
\]
where  $\tilde{S}=S-c(x_{1}^{d_{1}}x_{2}^{d_{2}}-1)$. The bordered Hessian  is the Hessian of  the function $S(x_1, x_2)-cV(x_1, x_2)$. 

The  critical points of $S\big|_{\mathcal{M}_{1}^{G}}$ are characterized as follows:  Let $H(g)$ be the value of $H$ at a critical point $g\in\mathcal{M}_{1}^{G}$, and let $\big|H(g)\big|$ denote  its determinant. If $\big|H(g)\big|>0$ then $g$ is a local maximum  of $S\big|_{\mathcal{M}_{1}^{G}}$, and if $\big|H(g)\big|<0$ then $g$ is a local minimum  of $S\big|_{\mathcal{M}_{1}^{G}}$. If  $ \big|H(g)\big|=0$  then   $g$ is a saddle point (cf. \cite{Ma}).

\begin{theorem}
 Let $M=G/K$ be a generalized flag manifold with $\fr{m}=\fr{m}_1\oplus\fr{m}_2$, and let $d_{i}= \dim\fr{m}_{i}$  $(i=1,2)$.  Then the $G$-invariant Einstein metrics of $M$ given in Theorem 2  are both local minima   of the scalar curvature functional $S$ restricted to the space of $G$-invariant metrics of volume one $\mathcal{M}_{1}^{G}$.
 \end{theorem}
\begin{proof}
   The volume of $g$ is $V=x_1^{d_1}x_2^{d_{2}}$, so
 \[
 -\displaystyle\frac{\partial V}{\partial x_1}=-d_1x_1^{d_1-1}x_{2}^{d_2}, \qquad  -\displaystyle\frac{\partial V}{\partial x_2}=-d_2x_{1}^{d_1}x_{2}^{d_2-1}.
 \]
   From equations  (9) we obtain that
  \begin{eqnarray*}
 \frac{\partial^2 \tilde{S}}{\partial x_1^2} &=& \frac{d_{1}}{x_1^3}-\frac{3tx_2}{2x_1^4}-cd_1(d_1-1)x_1^{d_1-2}x_2^{d_2},\\
\frac{\partial^2 \tilde{S}}{\partial x_2^2} &=&\frac{d_2-t}{x_2^3}-cd_2(d_2-1)x_1^{d_1}x_2^{d_2-2},\\
\frac{\partial^2\tilde{S}}{\partial x_1\partial x_2} &=& \frac{t}{2x_1^3}-cd_1d_2x_1^{d_{1-1}}x_{2}^{d_2-1},
\end{eqnarray*}
where $t=\displaystyle\frac{d_1d_2}{d_1+4d_2}$.

We first examine   the critical point $g=(1, 2)$, i.e. the K\"ahler-Einstein metric of $M$.  A computation gives that
\begin{equation}
\big|H(g)\big|=-(d_{1}+d_2)d_{1}d_{2}2^{2d_2-2}\Big(\frac{d_{2}}{d_1+4d_{2}}+c2^{d_{2}}\Big).
\end{equation}
Since the Einstein constant $c$ and the dimensions $d_{1}, d_{2}$ are  positive real numbers, it is   $\big|H(g)\big|<0$.  Thus the K\"ahler-Einstein metric   is a local minimum of   $S\big|_{\mathcal{M}_{1}^{G}}$.

For the second critical point $g=(1, \displaystyle\frac{4d_{2}}{d_{1}+2d_{2}})$, we obtain that  
\begin{equation}
\big|H(g)\big|=-d_1\big(\frac{4d_2}{d_1+2d_2}\big)^{2d_2-2}\Big(\frac{d_{1}^{3}d_{2}+5d_{1}^{2}d_{2}^{2}+6d_{1}d_{2}^{3}+2d_{2}^{4}}{(d_{1}+2d_{2})(d_{1}+4d_{2})}+ cd_{2}\big(\frac{4d_2}{d_1+2d_2}\big)^{d_2}(d_{1}+d_2)\Big),
\end{equation}
so $\big|H(g)\big|<0$,  and the non K\"ahler-Einstein metric is also a local minimum   of  $S\big|_{\mathcal{M}_{1}^{G}}$.   
\end{proof}

\begin{example}
\textnormal{Consider the space $E_{6}/SU(5)\times SU(2)\times U(1)$.  According to Table 3, it is $d_{1}=\dim\fr{m}_1=40$ and $d_{2}=\dim\fr{m}_{2}=10$, therefore $t=5$.  By Proposition 1, the scalar curvature is given by
\[
S=\frac{20}{x_{1}}+\frac{5}{2x_{2}}-\frac{5x_{2}}{4x_1^2}.
\]
From  (11) and
for the K\"ahler-Einstein metric $g=(1, 2)$  we obtain that 
\[
\big|H(g)\big|= -655360000 (1 + 8192 c)<0,
 \]
so $g$ is a local minimum of $S\big|_{\mathcal{M}_{1}^{G}}$. The non K\"ahler-Einstein metric is given by $g=(1, 2/3)$.  From (12)  it follows that
\[
\big|H(g)\big|= -11141120000/1162261467 - (53687091200000 c)/22876792454961<0, 
 \]
so it is also a local minimum   of $S\big|_{\mathcal{M}_{1}^{G}}$.}
\end{example}

 \end{document}